\documentclass[11pt,a4paper]{amsart}
\usepackage{dsfont}
\usepackage[applemac]{inputenc}
\usepackage{amsmath,amssymb}
\usepackage{hyperref}
\usepackage{ulem}
\usepackage{color}
\usepackage{epsfig}
\usepackage{graphicx,scalerel}
\usepackage{enumerate}
\usepackage{enumitem} 
\allowdisplaybreaks

\newtheorem{prop}{Proposition}
\newtheorem{cor}{Corollary}

\newtheorem{lem}{Lemma}
\newtheorem{theorem}{Theorem}
\newtheorem*{theorem*}{Theorem}

\def \R{\mathbb{R}}
\def \N{\mathbb{N}}
\def \Z{\mathbb{Z}}
\def \T{\mathbb{T}}
\def \I{\mathcal I}
\def \J{\mathcal J}
\def \C{\mathcal C}
\def \A{\mathcal A}
\def \B{\mathcal B}
\def \D{\mathcal D}
\def \F{\mathcal F}
\def \K{\mathcal K}

\def \P{\mathcal P}
\def \G{\mathcal G}

\newcommand{\eps}{\varepsilon}
\newcommand{\Aca}{\mathcal{A}}
\newcommand{\Bca}{\mathcal{B}}
\newcommand{\Fca}{\mathcal{F}}

\def \1{\textbf{1}}
\usepackage{setspace}
\newcommand{\setone}{\mathds{1}}
\newcommand{\E}[1]{\mathbb E\left[#1\right]}
\newcommand{\Ea}[1]{\mathbb E_1\left[#1\right]}
\newcommand{\Eb}[1]{\mathbb E_2\left[#1\right]}
\newcommand{\Ec}[1]{\mathbb E_3\left[#1\right]}

\newcommand{\pr}[1]{\left(#1\right)}
\newcommand{\ind}[1]{{\mathds{1}}_{#1}}
\newcommand{\PP}{\mathbb P}
\newcommand{\norm}[1]{\left\lVert #1\right\rVert}
\newcommand{\abs}[1]{\left\lvert #1\right\rvert}
\newcommand{\ens}[1]{\left\{#1\right\}}

\begin{document}
\pagestyle{plain}
\title{What can be the limit in the CLT for a field of martingale differences ?}

\author{Davide Giraudo, Emmanuel Lesigne, Dalibor Voln\'y}\date{\today}

\begin{abstract}
The now classical convergence in distribution theorem for well normalized sums of
stationary martingale increments has been extended to multi-indexed martingale
increments (see \cite{MR3913270} and references in there). In the present
article we make progress in the identification of the limit law.

In dimension one, as soon as the stationary martingale increments form an ergodic process, the limit law is normal, and it is still
the case for multi-indexed martingale increments when one of the processes defined by one coordinate of the
{\it multidimensional time} is ergodic. In the general case, the limit may be non normal.

The dynamical properties of the $\mathbb{Z}^d$-measure preserving action associated
to the stationary random field allows us to give a necessary and sufficient condition
for the existence of a non-normal limit law, in terms of entropy of some random processes.
The identification of a {\it natural} factor on which the $\mathbb{Z}^d$-action is {\it of product type}
is a crucial step in this approach.
\end{abstract}
\maketitle
\tableofcontents
\section{Introduction}

We study here limit theorems of CLT (Central Limit Theorem) type for
stationary multiparameters martingale indexed by $\Z^d$. In order to
limit the number of suspension points and make the text easier to
read, we choose $d=3$, but all what is said for this particular case
can be extended to any integer $d\geq2$. Some results are specific to
the case $d=2$ (and unknown for $d>2$), in which case they will be
stated (and proved) with $d=2$.\\

We consider a $\Z^3$ measure preserving action  $T=(T_{i,j,k})_{i,j,k\in\Z}$ on a probability space $(\Omega,\A,\mu)$, equipped with a {\it completely commuting invariant filtration} $\left(\F_{i,j,k}\right)_{i,j,k\in\Z}$, that is a family of sub-$\sigma$-algebras of $\A$ satisfying for all $(i,j,k)$ and $(i',j',k')$ in $\Z^3$,
\begin{enumerate}[label=(\roman*)]
\item \label{propriete_invariante} $\F_{i,j,k}=T_{-i,-j,-k}\F_{0,0,0}$ ;
\item\label{propriete_commutativite} For all integrable function
$f$,
$\E{\E{f\,|\,\F_{i,j,k}}\,|\,\F_{i',j',k'}}=\E{f\,|\,\F_{\min(i,
i'),\min(j,
j'),\min(k,k')}}$.
\end{enumerate}
Note that property \ref{propriete_commutativite} implies that $\F_{i,j,k}\cap\F_{i',j',k'}=\F_{\min(i,i'),\min(j,j'),\min(k,k')}$ and in particular $\F_{i,j,k}\subset\F_{i',j',k'}$ when $i\leq i'$, $j\leq j'$ and $k\leq k'$. 

We will use classical notations for the limit sub-$\sigma$ algebras when parameters go to $\pm\infty$ :
$$\F_{-\infty,j,k}=\bigcap_{i\in\Z}\F_{i,j,k}\quad\text{and}\quad\F_{\infty,j,k}=\bigvee_{i\in\Z}\F_{i,j,k},$$
and so on.

A {\it field of martingale differences} is a field of random variables $(X_{i,j,k})_{i,j,k\in\Z}$ of the type 
$$
X_{i,j,k}=f\circ T_{i,j,k}
$$
where $f\in  \mathbb{L}^2(\Omega,\F_{0,0,0},\mu)$ satisfies
\begin{equation}\label{eq:cond-mart}
\E{f\,|\,\F_{-1,\infty,\infty}}=\E{f\,|\,\F_{\infty,-1,\infty}}=\E{f\,|\,\F_{\infty,\infty,-1}}=0.
\end{equation}
(We will say simply that $f$ is a martingale difference adapted to the filtration $(\F_{i,j,k})$)\\

From a previous article \cite{MR3913270}, we know that, for any such field, we 
have convergence in law of
\begin{equation}
\label{eq:sommes_partielles_normalisees}
\frac1{\sqrt{\ell mn}}\sum_{i=1}^\ell\sum_{j=1}^m\sum_{k=1}^n X_{i,j,k}
\end{equation}
when $\min(\ell,m,n)$ goes to infinity.

Moreover the limit is normal as soon as one of the transformations $T_{1,0,0}$, 
$T_{0,1,0}$ or $T_{0,0,1}$ is ergodic (this was already established in  
\cite{MR3427925}), but some easy examples show that the limit is not normal in 
general. Indeed, following the example in \cite{MR3222815}, we can take 
$X_{i,j,k}=U_iV_iW_j$, where
$\pr{U_i}_{i\in\Z}$, $\pr{V_j}_{j\in\Z}$ and $\pr{W_k}_{k\in\Z}$
are three mutually independent i.i.d.\ sequences of standard normal random
variables, then for each
$\ell,m,n$, the random
variable defined by \eqref{eq:sommes_partielles_normalisees} has the
same distribution as the product of three independent random
variables having standard normal distribution. (Note that this example is produced by an ergodic $\Z^3$-action.)\\
Let us also mention the papers \cite{MR3504508} and \cite{MR3869881}, which bring further results and examples.

In all the sequel, we suppose that the $\Z^3$-action $T$ is ergodic on $(\Omega,\A,\mu)$. Moreover, we suppose that $\F_{\infty,\infty,\infty}=\A$ which does not cost anything since the whole process we are interested in is $\F_{\infty,\infty,\infty}$-measurable.

Here is the general organization of this article. 

In Section \ref{product-type factor} we describe a particular factor $\I$ on which the $\Z^3$-action is of {\it product type}. For an action of this type, the possible limit distributions of \eqref{eq:sommes_partielles_normalisees} are fully understood, as described in Section \ref{product-type CLT}.

In Section \ref{dim2}, where we restrict to the case of $\Z^2$-actions, we study what can happen on the orthocomplement of the factor $\I$. In particular we obtain the following results:

- if the transformation $T_{1,0}$ acting on the factor of $T_{0,1}$-invariants has zero entropy (or if the transformation $T_{0,1}$ acting on the factor of $T_{1,0}$-invariants has zero entropy), then for any square integrable martingale difference the limit distribution in the CLT is normal.

- if the transformation $T_{1,0}$ acting on the factor of $T_{0,1}$-invariants and the transformation $T_{0,1}$ acting on the factor of $T_{1,0}$-invariants have positive entropies, then there exists a square integrable martingale difference for which the limit distribution in the CLT is not normal.

\section{A factor of product type}\label{product-type factor}
Let us denote by $\I_1$, $\I_2$ and $\I_3$ the $\sigma$-algebras of, respectively $T_{1,0,0}$, $T_{0,1,0}$ and $T_{0,0,1}$ invariant sets in $\A$. 

Let us denote by $\overline\I_1$, $\overline\I_2$ and $\overline\I_3$ the $\sigma$-algebras of invariant sets under, respectively, $\Z^2$-actions $(T_{0,j,k})$, $(T_{i,0,k})$ and $(T_{i,j,0})$. In other words, we have
$$
\overline\I_1=\I_2\cap\I_3\ ,\quad\overline\I_2=\I_1\cap\I_3\quad\text{and}\quad \overline\I_3=\I_1\cap\I_2.
$$
(Note that, for an action of $\Z^d$ we would have set
$\overline\I_1=\I_2\cap\I_3\cap\dots\cap\I_d$, and that for the
particular case $d=2$, we have $\overline\I_1=\I_2$.)

Finally we note $\I$ the $\sigma$-algebra generated by the union of the $\overline\I_t$ :
$$
\I=\overline\I_1\vee\overline\I_2\vee\overline\I_3\;.
$$

Note that all these $\sigma$-algebras are invariant under the action $T$, hence they define factors of the dynamical system $(\Omega,\A,\mu,T)$.

\begin{prop}\label{indep}
The $\sigma$-algebras $\overline\I_1$, $\overline\I_2$ and $\overline\I_3$ are independent.
\end{prop}

As a consequence of this proposition we can state that the action $T$ on the probability space $(\Omega,\I,\mu)$ is of product type, which means that it is isomorphic to a $\Z^3$-action defined on the product of three probability spaces $(\Omega_t,\A_t,\mu_t)$ ($t=1,2,3$) by a formula of the type
$$
T_{i,j,k}(\omega_1,\omega_2,\omega_3)=(T_1^i\omega_1,T_2^j\omega_2,T_3^k\omega_3)
$$
where each $T_t$ is an invertible measure preserving transformation of $(\Omega_t,\A_t,\mu_t)$.

\begin{proof}[Proof of Proposition \ref{indep}]
Let $A_t\in\overline\I_t$, $t=1,2,3$. By ergodicity of the $\Z^3$-action, we have
$$
\lim_{\ell,m,n\to\infty}\frac1{\ell mn}\sum_{i=1}^\ell\sum_{j=1}^m\sum_{k=1}^n (\setone_{A_1}\setone_{A_2}\setone_{A_3})\circ T_{i,j,k}=\mu(A_1\cap A_2\cap A_3)
$$
but, by invariance,
\begin{multline*}
\frac1{\ell nm}\sum_{i=1}^\ell\sum_{j=1}^m\sum_{k=1}^n (\setone_{A_1}\setone_{A_2}\setone_{A_3})\circ T_{i,j,k}=\\\left(\frac1\ell\sum_{i=1}^\ell\setone_{A_1}\circ T_{i,0,0}\right)\left( \frac1m\sum_{j=1}^m\setone_{A_2}\circ T_{0,j,0}\right)\left( \frac1n\sum_{k=1}^n\setone_{A_3}\circ T_{0,0,k} \right)\\\longrightarrow\mu(A_1)\mu(A_2)\mu(A_3)\;,
\end{multline*}
since each of the three systems $(\Omega,\overline\I_1,\mu,T_{1,0,0})$, $(\Omega,\overline\I_2,\mu,T_{0,1,0})$ and $(\Omega,\overline\I_3,\mu,T_{0,0,1})$ is ergodic.
\end{proof}

\section{Martingale property preserved by projection}\label{proj-mart}
Our aim here is to show that the martingale property is preserved by projection on the factor $\I$. We begin by a general abstract result, which is stated for a $\Z^2$-measure preserving action but which can be proved similarly for any $\Z^{d-1}$-action.

Let $(X,\B,\nu)$ be a probability space and $U=(U_{i,j})$ a $\Z^2$-measure preserving action on this space. We denote by $\J$ the $\sigma$-algebra of $U$-invariant elements of the $\sigma$-algebra $\B$.

We denote by $\F$ a $U$-invariant sub-$\sigma$-algebra of $\B$, meaning that $U_{i,j}(\F)=\F$ for all $i,j\in\Z$~; we denote by $\C$ a sub-$\sigma$-algebra of $\F$.

\begin{prop}\label{pro}${}$

\begin{enumerate}[label=(\alph*)]
  \item\label{assertion_EfF_mea} For all $f\in  \mathbb{L}^2(\J\vee\C)$,
$\E{f\,|\,\F}$ is
$\J\vee\C$-measurable.
\item\label{assertion_F_mes_J_vee_C} For all $f\in  \mathbb{L}^2(\F)$,
$\E{f\,|\,\J\vee\C}$ is
$\F$-measurable.
\end{enumerate}
\end{prop}

\begin{cor}\label{co}
The conditional expectations with respect to $\F$ and to $\J\vee\C$ are commuting : for all $f\in  \mathbb{L}^2(\B)$,
$$
\E{\E{f\,|\,\J\vee\C}\,|\,\F}=\E{\E{f\,|\,\F}\,|\,\J\vee\C}=\E{f\,|\,\F\cap(\J\vee\C)}.
$$
\end{cor}

\begin{proof}[Proof of Corollary \ref{co}] 
Suppose that $\D$ and $\D'$ are two sub-$\sigma$-algebras such that, for all $f\in  \mathbb{L}^2(\D)$, $\E{f\,|\,\D'}$ is $\D$-measurable. By ordinary properties of projections, we have, for any $f\in  \mathbb{L}^2(\B)$
$$
\E{f\,|\,\D\cap\D'}=\E{\E{\E{f\,|\,\D}\,|\,\D'}\D\cap\D'}
$$
which implies
$$
\E{f\,|\,\D\cap\D'}=\E{\E{f\,|\,\D}\,|\,\D'}
$$
since $\E{\E{f\,|\,\D}\,|\,\D'}$ is $\D\cap\D'$-measurable.
\end{proof}
\begin{proof}[Proof of Proposition \ref{pro}]
Let $f\in  \mathbb{L}^2(\F)$ ; by the ergodic theorem, we have $$
\E{f\,|\,\J}=\lim_{\ell,m\to\infty}\frac1{\ell m}\sum_{i=1}^\ell \sum_{j=1}^mf\circ U_{i,j}\;,$$
hence $\E{f\,|\,\J}$ is $\F$-measurable, since $\F$ is $U$-invariant.

For any $f\in  \mathbb{L}^2(\B)$, we have
$$
\E{f\,|\,\F}\circ U_{i,j}=\E{f\circ U_{i,j}\,|\, U_{i,j}^{-1}(\F)}=\E{f\circ U_{i,j}\,|\,\F},
$$
hence if $f$ is $\J$-measurable, then $\E{f\,|\,\F}$ is $\J$-measurable.

Consider now $g\in  \mathbb{L}^2(\J)$ and $h\in  \mathbb{L}^2(\C)$. 
Since $\C\subset \F$, we have
$ 
\E{gh\,|\,\F}=\E{g\,|\,\F} h
$, so $\E{gh\,|\,\F}$ is $\J\vee\C$-measurable.
But the functions of the form $gh$ with $g\in  \mathbb{L}^\infty(\J)$ and $h\in  \mathbb{L}^\infty(\C)$ generate a dense subspace of $ \mathbb{L}^2(\J\vee\C)$, so that assertion \ref{assertion_EfF_mea} is proved.\\

Consider now $f\in  \mathbb{L}^2(\F)$, $g\in  \mathbb{L}^2(\J)$ and $h\in  \mathbb{L}^2(\C)$. We have
\begin{align*}
\langle \E{f\,|\,\J\vee\C},gh\rangle &=\langle f,gh\rangle\quad\text{(because $gh$ is $\J\vee\C$-measurable)},\\
&=\langle f,\E{gh\,|\,\F}\rangle\quad\text{(because $f$ is $\F$-measurable)},\\
&=\langle f,h\E{g\,|\,\F}\rangle\quad\text{(because $h$ is $\F$-measurable)},\\
&=\langle \E{f\,|\,\J\vee\C},h\E{g\,|\,\F}\rangle \quad\text{(we know that $\E{g\,|\,\F}$ is $\J$-measurable)},\\
&=\langle \E{f\,|\,\J\vee\C},\E{gh\,|\,\F}\rangle.
\end{align*}
By the density argument, we conclude that, for all $k\in  \mathbb{L}^2(\J\vee\C)$, $$\langle \E{f\,|\,\J\vee\C},k)\rangle=\langle \E{f\,|\,\J\vee\C},\E{k\,|\,\F}\rangle.$$ This identity applied to the function $k= \E{f\,|\,\J\vee\C}$ shows that this function is $\F$-measurable. This is assertion \ref{assertion_F_mes_J_vee_C}.
\end{proof}

Another result we need is the following classical lemma and we give a short proof for the sake of completeness.
\begin{lem} \label{class} Let $S$ be a measure preserving transformation of the probability space $(X,\B,\nu)$, and $\F_n = S^{-n}\F_0$ be an increasing filtration in $\B$. Denote by $\K$ the sub-$\sigma$-algebra of $S$ invariant sets. Then $\K \cap \F_{\infty}=\K \cap \F_{-\infty}$.
In particular, if $\K \subset \F_{\infty}$ then $\K \subset \F_{-\infty}$.
\end{lem}

\begin{proof}[Proof of Lemma \ref{class}]
 Let $f\in  \mathbb{L}^2(\nu)$. By purely Hilbert space arguments, we know that, in the space $ \mathbb{L}^2(X)$,
 $$
 \lim_{n\to-\infty}\E{f\,|\,\F_n}=\E{f\,|\,\F_{-\infty}}\quad\text{and}\quad  \lim_{n\to\infty}\E{f\,|\,\F_n}=\E{f\,|\,\F_{\infty}}\;.
 $$ 
Now suppose that $f$ is invariant under $S$, that is $f$ is $\K$-measurable. Then
$$
\E{f\,|\,\F_n}=\E{f\circ S^n\,|\, S^{-n}\F_0}=\E{f\,|\,\F_0 }\circ S^n.
$$
Thus $\|\E{f\,|\,\F_{-m}}\|_2=\|\E{f\,|\,\F_n}\|_2$ and with $n,m\to\infty$, we obtain $\|\E{f\,|\,\F_{-\infty}}\|_2=\|\E{f\,|\,\F_\infty}\|_2$.
Since $\F_{-\infty}\subset\F_\infty$, this implies 
 $
\E{f\,|\,\F_{-\infty}}=\E{f\,|\,\F_{\infty}}
$.
In particular, if $f$ is $\F_{\infty}$-measurable, then it is $\F_{-\infty}$-measurable, which is what we had to prove.
\end{proof}

We now come back to the situation described in the preceding section where the factor $\I$ is defined, and here is the result we were looking for.

\begin{theorem}
Let $f$ be a martingale difference adapted to the filtration $(\F_{i,j,k})$. Then its projection $\E{f\,|\,\I}$ is a martingale difference adapted to the filtration $(\F_{i,j,k})$, as well adapted to the filtration $(\I\cap\F_{i,j,k})$. Moreover $f-\E{f\,|\,\I}$ is also a martingale difference adapted to the filtration $(\F_{i,j,k})$.
\end{theorem}

\begin{proof}
We want to apply Proposition \ref{pro} to the $\Z^2$-action $U_{j,k}=T_{0,j,k}$ on the space $(\Omega,\A,\mu)$. So we have $\J=\overline \I_1$. We consider also $\C=\overline I_2\vee\overline \I_3$ and $\F=\F_{i,\infty,\infty}$, for a given $i$. We know that $\F$ is $U$-invariant and, thanks to Lemma \ref{class}, we have $\I_1\subset\F$, hence $\C\subset\F$ since $\C\subset\I_1$.

Note that $\J\vee\C=\overline \I_1\vee\overline \I_2\vee\overline \I_3=\I$. 

Now Corollary \ref{co} tells us that the conditional expectation with respect to $\I$ commutes with the conditional expectations with respect to $\F_{i,\infty,\infty}$. Of course, we can exchange the roles of $i$, $j$ and $k$ and we find as well that the conditional expectation with respect to $\I$ commutes with the conditional expectations with respect to $\F_{\infty,j,\infty}$ and with the conditional expectations with respect to $\F_{\infty,\infty,k}$.

By the complete commuting property of the filtration, we have
$$
\E{\cdot\,|\,\F_{i,j,k}}=\E{\E{\E{\cdot\,|\,\F_{i,\infty,\infty}}\,|\,\F_{\infty,j,\infty}}\,|\,\F_{\infty,\infty,k}},
$$
and we conclude that the conditional expectation with respect to $\I$ commutes with the conditional expectations with respect to $\F_{i,j,k}$ : for all $f\in  \mathbb{L}^2(\A)$, for all $i,j,k\in\Z\cup\{\infty\}$,
\begin{equation}\label{eq:commut}
\E{\E{f\,|\,\F_{i,j,k}}\,|\,\I}=\E{\E{f\,|\,\I}\,|\,\F_{i,j,k}}=\E{f\,|\,\F_{i,j,k}\cap\I}.
\end{equation}

The first thing we want to see now is that $(\I\cap\F_{i,j,k})$ is a completely commuting invariant filtration. The first point is
$$T_{-i,-j,-k}(\I\cap\F_{0,0,0})=\I\cap\F_{i,j,k}$$
which is true thanks to (i) and the fact that $\I$ is invariant.
The second point is that, for all integrable function $f$, $$\E{\E{f\,|\,\I\cap\F_{i,j,k}}\,|\,\I\cap\F_{i',j',k'}}=\E{f\,|\,\I\cap\F_{\min(i,i'),\min(j,j'),\min(k,k')}}\;,$$
which, thanks to \eqref{eq:commut}, can be written
$$\E{\E{\E{f\,|\,\F_{i,j,k}}\,|\,\F_{i',j',k'}}\,|\,\I}=\E{\E{f\,|\,\F_{\min(i,i'),\min(j,j'),\min(k,k')}}\,|\,\I}$$
and is true thanks to (ii).

The second thing to verify is that if $f$ satisfies the martingale condition \eqref{eq:cond-mart} then $\E{f\,|\,\I}$ satisfies it. The facts that conditional expectations with respect to $\I$ and $\F_{0,0,0}$ commute and that $f$ is $\F_{0,0,0}$-measurable imply that $\E{f\,|\,\I}$ is  $\I\cap\F_{0,0,0}$-measurable. Moreover we have, thanks to \eqref{eq:commut}
and \eqref{eq:cond-mart}
$$
\E{\E{f\,|\,\I}\,|\,\F_{-1,\infty,\infty}}=\E{\E{f\,|\,\F_{-1,\infty,\infty}}\,|\,\I}=0
$$
and similarly
$$
\E{\E{f\,|\,\I}\,|\,\I\cap\F_{-1,\infty,\infty}}=0.
$$

We conclude that $\E{f\,|\,\I}$ is a martingale difference (for any of the two filtrations), and $f-\E{f\,|\,\I}$ is also a martingale difference adapted to the filtration $(\F_{i,j,k})$, just as difference of two martingale differences.
\end{proof}

\section{Limit law in the case of a product type action}\label{product-type CLT}
In this section, we deal with the case where $\pr{\Omega,\Aca,\mu}$
has the
form $\Omega=\Omega_1\times\Omega_2\times\Omega_3$, 
$\Aca=\Aca_1\otimes\Aca_2\otimes\Aca_3$ and 
$\mu=\mu_1\otimes\mu_2\otimes\mu_3$, where $\Aca_1$ (respectively 
$\Aca_2$, $\Aca_3$) is a $\sigma$-algebra on $\Omega_1$ (respectively 
$\Omega_2$, $\Omega_3$) and $\mu_1$, $\mu_2$, $\mu_3$ are probability
measures.
We consider an action $T$ of $\Z^3$ on $\Omega$ given by 
\begin{equation*}
 T_{i,j,k}\pr{\omega_1,\omega_2,\omega_3}=
 \pr{T_1^{i}\omega_1,T_2^{j}\omega_2,T_3^{k}\omega_3},\quad
i,j,k\in\Z.
\end{equation*}
We assume that the action of $T_1$ on $\Omega_1$ is ergodic, as well
as that of
$T_2$ on $\Omega_2$ and $T_3$ on $\Omega_3$.
For $\ell\in\{1,2,3\}$, consider a sub-$\sigma$-algebra
$\Fca_0^{\pr{\ell}}$ of
$\Aca_\ell$ such that 
$T_\ell\Fca_0^{\pr{\ell}}\subset \Fca_0^{\pr{\ell}}$. Define 
the sub-$\sigma$-algebra $\Fca_{i,j,k}$ of $\Aca$ by
\begin{equation*}
 \Fca_{i,j,k}=\pr{T_1^{-i}\Fca_0^{\pr{1}}}\otimes 
\pr{T_2^{-j}\Fca_0^{\pr{2}}}\otimes \pr{T_3^{-k}\Fca_0^{\pr{3}}}.
\end{equation*}

\begin{prop}
 The filtration $\pr{\Fca_{i,j,k}}_{i,j,k\in\Z}$ is a completely
commuting
invariant filtration.
\end{prop}

\begin{proof}
 Invariance follows by construction. Let us show commutativity
property, \ref{propriete_commutativite}. Replacing $f$ by
$g=\E{f\mid\Fca_{i,j,k}}$,
it suffices to show that for each integrable and
$\Fca_{i,j,k}$-measurable
function $g$, 
\begin{equation}\label{eq:commut_avec_g}
 \E{g\mid\Fca_{i',j',k'}}=\E{g\mid\Fca_{\min(i,i'),\min(j,
j'),\min(k,k')}}.
\end{equation}

Since the set of linear combinations of products of indicator
functions of the
form 
$\pr{\omega_1,\omega_2,\omega_3}\mapsto \mathds{1}_{A_1}\pr{\omega_1}
\mathds{1}_{A_2}\pr{\omega_2}
\mathds{1}_{A_3}\pr{\omega_3}$ with $A_1\in T_1^{-i}\Fca_0^{\pr{1}}$,
$A_2\in
T_2^{-j}\Fca_0^{\pr{2}}$ and $A_3\in T_3^{-k}\Fca_0^{\pr{3}}$ is dense in 
$\mathbb L^1\pr{\Fca_{i,j,k}}$, it suffices to prove
\eqref{eq:commut_avec_g}
when $g$ is of this form.
We use the following notation for
 $f_\ell\in\mathbb L^2\pr{\Omega_\ell,\mathcal
	 A_\ell,\mu_\ell}$, $\ell\in\ens{1,2,3}$,
\begin{equation*}
f_1\otimes f_2 \otimes f_3\pr{\omega_1,\omega_2,\omega_3}=
f_1\pr{\omega_1}f_2\pr{\omega_2}f_3\pr{\omega_3}.
\end{equation*}
One can check that if $\Bca_\ell$, $\ell\in\ens{1,2,3}$, are
sub-$\sigma$-algebras of $\Aca_\ell$, then
\begin{equation}\label{eq:cond_expectation_product}
\E{ f_1\otimes f_2\otimes f_3\mid  \Bca_1\otimes \Bca_2\otimes \Bca_3}
= \Ea{f_1\mid \Bca_1}\otimes \Eb{f_2\mid \Bca_2}\otimes \Ec{f_3\mid
\Bca_3}.
\end{equation}
Indeed, it suffices to check \eqref{eq:cond_expectation_product} when
$f_\ell$ is the indicator of a set $A_\ell$ of $\Aca_\ell$. For
$B_\ell\in\Bca_\ell$, let us write $B=B_1\times B_2\times B_3$. Using
independence, we have
\begin{align*}
\E{\pr{ \mathds{1}_{A_1}\otimes \mathds{1}_{A_2}\otimes
\mathds{1}_{A_3}  }\mathds{1}_{B} }&=
 \Ea{\mathds{1}_{A_1}\mathds{1}_{B_1} }\cdot
 \Eb{\mathds{1}_{A_2}\mathds{1}_{B_2} }\cdot
 \Ec{\mathds{1}_{A_3}\mathds{1}_{B_3} }\\
&= \Ea{\Ea{\mathds{1}_{A_1}\mid\Bca_1}\mathds{1}_{B_1} }\cdot
\Eb{\Eb{\mathds{1}_{A_2}\mid\Bca_2}\mathds{1}_{B_2} }\cdot
\Ec{\Ec{\mathds{1}_{A_3}\mid\Bca_3}\mathds{1}_{B_3} }\\
&=\E{ \Ea{\mathds{1}_{A_1}\mid\Bca_1}\cdot\mathds{1}_{B_1}\cdot
\Eb{\mathds{1}_{A_2}\mid\Bca_2}\cdot\mathds{1}_{B_2}\cdot
\Ec{\mathds{1}_{A_3}\mid\Bca_3}\cdot\mathds{1}_{B_3}  }\\
&=\E{
 \Ea{\mathds{1}_{A_1}\mid\Bca_1}\cdot
 \Eb{\mathds{1}_{A_2}\mid\Bca_2}\cdot
 \Ec{\mathds{1}_{A_3}\mid\Bca_3}\cdot\mathds{1}_{B }}
\end{align*}
which shows \eqref{eq:cond_expectation_product}.
When $g=\mathds{1}_{A_1}\otimes \mathds{1}_{A_2}\otimes
\mathds{1}_{A_3}$ with $A_1\in T_1^{-i}\Fca_0^{\pr{1}}$,
$A_2\in
T_2^{-j}\Fca_0^{\pr{2}}$ and $A_3\in T_3^{-k}\Fca_0^{\pr{3}}$,
$\Bca_1=T_1^{-i'}\Fca_0^{\pr{1}}$, $\Bca_2=T_2^{-j'}\Fca_0^{\pr{2}}$
and $\Bca_3=T_3^{-k'}\Fca_0^{\pr{3}}$,
\eqref{eq:cond_expectation_product} gives exactly
\eqref{eq:commut_avec_g}.
\end{proof}

In order to investigate the convergence of the partial sum process given by \eqref{eq:sommes_partielles_normalisees}, 
we will need to decompose the considered function $f$ as a sum of functions which can be expressed as a product of functions of a
single $\omega_\ell$.
\begin{lem}\label{lem:base_Hilbert}
Denote by $\Delta$ the set of square integrable functions satisfying the martingale property defined as in \eqref{eq:cond-mart}, that is 
\begin{equation*}
\Delta=\ens{f\in\mathbb L^2, f\mbox{ is }\Fca_{0,0,0}\mbox{-measurable and }\E{f\mid\Fca_{-1,\infty,\infty}}=\E{f\mid\Fca_{\infty,-1,\infty}}=\E{f\mid\Fca_{\infty,\infty,-1}}=0}.
\end{equation*}
 There exist random variables $v_{a,1}$, $v_{b,2}$, $v_{c,3}$,
$a,b,c\geq 1$,  such that the collection of random variables
 $\pr{\omega_1,\omega_2,\omega_3}\mapsto 
v_{a,1}\pr{\omega_1}v_{b,2}\pr{\omega_2}v_{c,3}\pr{\omega_3}$ is a
Hilbert basis of the space 
 $\Delta$.
\end{lem}
 \begin{proof}[Proof of Lemma~\ref{lem:base_Hilbert}]
The space $\mathbb L^2\pr{\Fca_{0}^{\pr{1}}}\ominus  
\mathbb L^2\pr{\Fca_{-1}^{\pr{1}}}$ is separable and admits 
a Hilbert basis $\pr{v_{a,1}}_{a\geq 1}$. Similarly,
we denote by $\pr{v_{b,2}}_{b\geq 1}$ a Hilbert basis of
$\mathbb L^2\pr{\Fca_{0}^{\pr{2}}}\ominus  
\mathbb L^2\pr{\Fca_{-1}^{\pr{2}}}$ and by $\pr{v_{c,3}}_{c\geq 1}$ a
Hilbert basis of
$\mathbb L^2\pr{\Fca_{0}^{\pr{3}}}\ominus  
\mathbb L^2\pr{\Fca_{-1}^{\pr{3}}}$.

We check that the collection of maps
$u_{a,b,c}\pr{\omega_1,\omega_2,\omega_3}\mapsto
v_{a,1}\pr{\omega_1}v_{b,2}\pr{\omega_2}v_{c,3}\pr{\omega_3}$ is a
Hilbert basis of the space 
 $\Delta$. 
 Orthonormality follows from the product structure of $\Omega$. We
 have to show that if $g\in \Delta$ is such that $\E{g \cdot u_{a,b,c}}=0$
for each
 $a,b,c\geq 1$, then $g=0$. To this aim, we write the previous
expectation as an integral over
$\Omega_1\times\Omega_2\times\Omega_3$
 and use Fubini's theorem to get that 
 \begin{equation*}
 \int_{\Omega_3}v_{c,3}\pr{\omega_3}\pr{
\int_{\Omega_1\times\Omega_{2}}
   v_{a,1}\pr{\omega_1}v_{b,2}\pr{\omega_1}
  g\pr{\omega_1,\omega_2,\omega_3}
  \mathrm{d}\mu_1\pr{\omega_1}\mathrm{d}\mu_{2}\pr{\omega_{2}}}\mathrm{d}\mu_3\pr{\omega_3}=0.
 \end{equation*}
Since $\pr{v_{ c,3}}_{c\geq 1}$ is a Hilbert basis of $\mathbb
L^2\pr{\Fca_{0}^{\pr{3}}}\ominus
\mathbb L^2\pr{\Fca_{-1}^{\pr{3}}}$, it follows that there exists $\Omega'_3\subset \Omega_3$ such that $\mu_3\pr{\Omega'_3}=1$ and for each 
$\omega_3\in\Omega'_3$, the equality 
\begin{equation*}
\int_{\Omega_1\times\Omega_{2}}
   v_{a,1}\pr{\omega_1}v_{b,2}\pr{\omega_1}
  g\pr{\omega_1,\omega_2,\omega_3}
  \mathrm{d}\mu_1\pr{\omega_1}\mathrm{d}\mu_{2}\pr{\omega_{2}}=0
\end{equation*}
takes place; doing the same reasoning gives sets $\Omega'_1\subset\Omega_1$, $\Omega'_2\subset\Omega_2$ such that 
$\mu_1\pr{\Omega'_1}=\mu_2\pr{\Omega'_2}=1$ and for $\omega_1\in\Omega'_1$, $\omega_2\in\Omega'_2$,  $g\pr{\omega_1,\omega_2,\omega_3}=0$ hence 
$g=0$ $\mu$-a.s.. This ends the proof of Lemma~\ref{lem:base_Hilbert}.
 \end{proof}

\begin{theorem}\label{thm:TLC_product}
Let $\pr{\Omega,\Aca,\mu}$ be a dynamical system of the form $\Omega=\Omega_1\times\Omega_2\times\Omega_3$, 
$\Aca=\Aca_1\otimes\Aca_2\otimes\Aca_3$ and 
$\mu=\mu_1\otimes\mu_2\otimes\mu_3$, where $\Aca_1$ (respectively 
$\Aca_2$, $\Aca_3$) is a $\sigma$-algebra on $\Omega_1$ (respectively 
$\Omega_2$, $\Omega_3$) and $\mu_1$, $\mu_2$, $\mu_3$ are probability
measures. For $\ell\in\{1,2,3\}$, consider a sub-$\sigma$-algebra
$\Fca_0^{\pr{\ell}}$ of
$\Aca_\ell$ such that 
$T_\ell\Fca_0^{\pr{\ell}}\subset \Fca_0^{\pr{\ell}}$. Define 
the sub-$\sigma$-algebra $\Fca_{i,j,k}$ of $\Aca$ by
\begin{equation*}
 \Fca_{i,j,k}=\pr{T_1^{-i}\Fca_0^{\pr{1}}}\otimes 
\pr{T_2^{-j}\Fca_0^{\pr{2}}}\otimes \pr{T_3^{-k}\Fca_0^{\pr{3}}}.
\end{equation*}
Let $f$ be function such that $\pr{f\circ T_{i,j,k}}_{i,j,k\in\Z}$ is 
a martingale difference random field. 

There exist a family of real numbers
$\pr{\lambda_{a,b,c}\pr{f}}_{a,b,c\geq 1}$
such that $\sum_{a,b,c\geq 1}\lambda_{a,b,c}^2\pr{f}<\infty$ and such that if $\pr{N_a^{\pr{1}}}_{a\geq 1}$, $\pr{N_b^{\pr{2}}}_{b\geq
1}$ and $\pr{N_c^{\pr{3}}}_{c\geq 1}$ are three
i.i.d.\
and mutually independent sequences of standard normal random
variables, then
\begin{equation*}
 \frac1{\sqrt{\ell mn}}\sum_{i=1}^\ell\sum_{j=1}^m\sum_{k=1}^n
 f\circ T_{i,j,k}\to \sum_{a,b,c=1}^\infty \lambda_{a,b,c}\pr{f}
 N_a^{\pr{1}}N_b^{\pr{2}}N_c^{\pr{3}}
\end{equation*}
in distribution as $\min\ens{\ell,m,n}\to\infty$.
\end{theorem}

It will be clear from the proof of this theorem that any square
summable family $\pr{\lambda_{a,b,c}\pr{f}}_{a,b,c\geq 1}$ can appear
in the expression of the limit in distribution.

\begin{proof}
  We know by Lemma~\ref{lem:base_Hilbert} that we can express $f$ as
  $$f\pr{\omega_1,\omega_2,\omega_3}=\sum_{a,b,c=1}^\infty
\lambda_{a,b,c}\pr{f}
v_{a,1}\pr{\omega_1}v_{b,2}\pr{\omega_2}v_{c,3}\pr{\omega_3},$$
where the convergence takes place in $\mathbb L^2\pr{\mu}$.

Define
\begin{equation*}
 f_K\pr{\omega_1,\omega_2,\omega_3}:=\sum_{a,b,c=1}^K
\lambda_{a,b,c}\pr{f}
v_{a,1}\pr{\omega_1}v_{b,2}\pr{\omega_2}v_{c,3}\pr{\omega_3}.
\end{equation*}
Note that $\pr{f_K\circ T^{i,j,k}}_{i,j,k\in\Z}$ is also a martingale
difference random field. Suppose that we proved for each $K\geq 1$
Theorem~\ref{thm:TLC_product} with $f$ replaced by $f_K$.
By orthogonality of increments, for all $\ell, m, n>0$,
$$
\PP\pr{\abs{\frac1{\sqrt{\ell
mn}}\sum_{i=1}^\ell\sum_{j=1}^m\sum_{k=1}^n
 f\circ T_{i,j,k}-\frac1{\sqrt{\ell
mn}}\sum_{i=1}^\ell\sum_{j=1}^m\sum_{k=1}^n
f_K\circ T_{i,j,k}}>\varepsilon}
\leq\frac
1{\varepsilon^2}\norm{f-f_K}_2^2
$$
hence
$$
  \lim_{K\to\infty}\sup_{\ell,m,n}
  \PP\pr{\abs{\frac1{\sqrt{\ell
mn}}\sum_{i=1}^\ell\sum_{j=1}^m\sum_{k=1}^n
 f\circ T_{i,j,k}-\frac1{\sqrt{\ell
mn}}\sum_{i=1}^\ell\sum_{j=1}^m\sum_{k=1}^n
f_K\circ T_{i,j,k}}>\varepsilon}=0.
$$
Moreover,
\begin{multline*}
  \lim_{K\to\infty}\norm{\sum_{a,b,c=1}^\infty \lambda_{a,b,c}\pr{f}
 N_a^{\pr{1}}N_b^{\pr{2}}N_c^{\pr{3}}-\sum_{a,b,c=1}^K
\lambda_{a,b,c}\pr{f}
 N_a^{\pr{1}}N_b^{\pr{2}}N_c^{\pr{3}}}_2^2\\ =
\lim_{K\to\infty}\norm{\sum_{a,b,c=1}^\infty
\mathds{1}_{\max\ens{a,b,c}\geq K+1  }\;
  \lambda_{a,b,c}^2\pr{f}
}_2^2=0
\end{multline*}
hence we would get the conclusion of  Theorem~\ref{thm:TLC_product}
by an application of Theorem~4.2 in \cite{MR0233396}.
We thus have to prove that for each $K$,
\begin{equation}\label{eq:tlc_produit_somme_de_K_elements}
 \frac1{\sqrt{\ell mn}}\sum_{i=1}^\ell\sum_{j=1}^m\sum_{k=1}^n
 f_K\circ T_{i,j,k}\to \sum_{a,b,c=1}^K\lambda_{a,b,c}\pr{f}
 N_a^{\pr{1}}N_b^{\pr{2}}N_c^{\pr{3}}
\end{equation}
in distribution as $\min\ens{\ell,m,n}\to\infty$.
By definition of $f_K$,
\begin{multline*}
  \frac1{\sqrt{\ell mn}}\sum_{i=1}^\ell\sum_{j=1}^m\sum_{k=1}^n
  f_K\circ T_{i,j,k}\\
  =\sum_{a,b,c=1}^K\lambda_{a,b,c}\pr{f}
    \pr{\frac 1{\sqrt{\ell}}\sum_{i=1}^\ell v_{a,1}\circ T_1^i}
      \pr{\frac 1{\sqrt{m}}\sum_{j=1}^m v_{b,2}\circ T_2^j}
        \pr{\frac 1{\sqrt{n}}\sum_{k=1}^n v_{c,3}\circ
T_3^k}.
\end{multline*}
Consider the random vector $V_{\ell,m,n}$ of dimension $3K$, where
the first $K$ entries are $\frac 1{\sqrt{\ell}}\sum_{i=1}^\ell
v_{a,1}\circ T_1^i$, $1\leq a\leq K$, the entries
of index between $K+1$ and $2K$ are $\frac 1{\sqrt{m}}\sum_{j=1}^m
v_{b,2}\circ T_2^j$ and the last $K$ are $\frac
1{\sqrt{n}}\sum_{k=1}^n v_{c,3}\circ
T_3^k$. By the Cramer-Wold device, the Billingsley-Ibragimov Central Limit Theorem for martingale differences and the fact that
$\norm{v_{a,1}}_2=\norm{v_{b,2}}_2=\norm{v_{c,3}}_2=1$, the
vector $V_{\ell,m,n}$ converges in distribution as
$\min\ens{\ell,m,n}\to\infty$ to $$V:=\pr{N_1^{\pr{1}},\dots,N_K^{\pr{1}},
N_1^{\pr{2}},\dots,N_K^{\pr{2}},N_1^{\pr{3}},\dots,N_K^{\pr{3}}}\;,$$
where $N_a^{\pr{1}}$, $N_b^{\pr{2}}$ and $N_c^{\pr{3}}$ are like in
the statement of Theorem~\ref{thm:TLC_product}.
Now, \eqref{eq:tlc_produit_somme_de_K_elements} follows from an
application of the continuous mapping theorem, that is,
$g\pr{V_{\ell,m,n}}\to g\pr{V}$, where $g\colon \R^{3K}\to \R$ is
defined as
\begin{equation*}
  g\pr{x_1,\dots,x_K,y_1,\dots,y_K,z_1,\dots,z_K}
  =\sum_{a,b,c=1}^K\lambda_{a,b,c}\pr{f} x_ay_bz_c.
\end{equation*}
This ends the proof of Theorem~\ref{thm:TLC_product}.
\end{proof}

\section{Limit law in the general case, for 2-dimensional field}\label{dim2}
As shown in \cite{MR3913270}, for a random field of martingale differences we 
have a CLT with convergence towards a mixture
of normal laws (see Theorem below).
In Section \ref{product-type CLT}, for $f$ $\I$-measurable it was precised which mixtures can appear as limit laws (for definition
of the factor $\I$ see Section~\ref{product-type factor}).
Here we deal with the same question for the general case of a martingale difference $f\in  \mathbb{L}^2$.
We reduce our study to the case of (ergodic) $\Z^2$-actions. In many cases it is because we have not succeeded
to extend the proofs to $d>2$.

As shown in Section \ref{proj-mart}, if $f$ is a martingale difference, so is $\E{f | \I}$ and also
$f - \E{f | \I}$. Recall that the limit laws for the random field generated by $\E{f | \I}$ have been determined in Section \ref{product-type CLT}).
In Subsection \ref{queue-property}, we give a sufficient condition guaranteeing convergence of the random field generated by $f - \E{f | \I}$ to a normal law.
We show in Subsection \ref{convolution} that, under the same condition the random field generated by $f$ is the convolution of the preceding ones. 

In Subsection \ref{entropie} we establish the result announced at the
end of the Introduction.

Eventually in Subsection \ref{nnormal_ex} we give
an example of a field of martingale differences generated by  $f - \E{f | \I}$ where the limit law is not normal.
It remains an open question which mixtures of normal laws can appear as limits in the CLT for $f - \E{f | \I}$.

In all this Section, $f\circ T_{i,j}$ is a field of martingale
differences adapted to a completely commuting filtration $\F_{i,j}$.

\subsection{Limit law for an increment orthogonal to the factor of product type}\label{queue-property}

Recall that $f\circ T_{i,j}$ is a field of martingale differences and
as shown in Part 1, $(f - \E{f | \I})\circ T_{i,j}$ are
martingale differences as well.

Let us begin by recalling Theorem 1 in \cite{MR3913270} which gives information 
on the limit law in the CLT. It will be stated and used here for $d=2$ but 
extends to any dimension.
\begin{theorem*} 
When $\min\{m,n\}\to\infty$ the random variables $\frac1{\sqrt{mn}}\sum_{i=1}^m\sum_{j=1}^nf\circ T_{i,j}$ converge in distribution to a law with characteristic function $\E{\exp(-\eta^2t^2/2)}$ where $\eta^2$ is a positive random variable such that $\E{\eta^2}=\norm{f}^2$. The random variables $\frac1{mn}\sum_{i=1}^m\left(\sum_{j=1}^n f\circ T_{i,j}\right)^2$ converge in distribution to $\eta^2$.
\end{theorem*}

\underline{Comment on the Theorem.} In fact, by the ergodic theorem, $\lim_{m\to\infty}\frac1{m}\sum_{i=1}^m\frac1n\left(\sum_{j=1}^n f\circ T_{i,j}\right)^2$ exists for each $n$ and the distribution of $\eta^2$ is the limit in distribution of this quantity when $n\to\infty$. (This can be seen as well in the proof of the Theorem or as a consequence of it.)

\begin{prop}\label{queue-condition}
Let $f\in \mathbb{L}^2\ominus \mathbb{L}^2(\I)$ be a martingale
difference. If, moreover, $f\in\mathbb{L}^4$ and
\begin{equation}\label{eq:hyp_pour_avoir_loi_normale}
  \lim_{\ell\to\infty}
  \norm{
    \E{f\mid\F_{\infty, -\ell}\vee \I_1 }}_2=0.
\end{equation}
then for $\min\{m,n\} \to\infty$, $(1/\sqrt{mn}) \sum_{i=1}^m\sum_{j=1}^n f\circ 
T_{i,j}$ converge in distribution
to a centered normal law with variance $\E{f^2}$.
\end{prop}

\underbar{Remark 1.} We have $\I_2\subset \F_{\infty, -\ell}$ (for every $\ell$) hence $\I \subset
\F_{\infty, -\ell}\vee \I_1$. Therefore $\|\E{f\,|\,\I}\|_2 \leq \|\E{f\,|\, \F_{\infty, -\ell}\vee \I_1}\|_2$
and, for any $f\in  \mathbb{L}^2$,  \eqref{eq:hyp_pour_avoir_loi_normale} implies $\E{f\,|\,\I} =0$. \\

\underbar{Remark 2.} For any $f\in  \mathbb{L}^2$, the sequence $\left(\left(\E{f\,|\, \F_{\infty, -\ell}\vee \I_1}\right)^2\right)$ is uniformly integrable ; indeed, denoting $\left(\G_\ell\right)$ a family of sub-$\sigma$-algebra we have
$$
\E{\left(\E{f\mid\G_\ell}\right)^2\ind{\left|\E{f\mid\G_\ell}\right|\geq C}}\leq\E{\E{f^2\mid\G_\ell}\ind{\left|\E{f\mid\G_\ell}\right|\geq C}}=\E{f^2\ind{\left|\E{f\mid\G_\ell}\right|\geq C}}\to 0
$$
when $C\to\infty$. \\As a consequence the property \eqref{eq:hyp_pour_avoir_loi_normale} is equivalent to the convergence of  $\left(\E{f\mid\F_{\infty, -\ell}\vee \I_1}\right)$ to zero in probability.

Similarly, for $f\in  \mathbb{L}^4$, condition  \eqref{eq:hyp_pour_avoir_loi_normale} implies that
\begin{equation}\label{eq:hyp_pour_avoir_loi_normale_4}
  \lim_{\ell\to\infty}
  \norm{
    \E{f\mid\F_{\infty, -\ell}\vee \I_1 }}_4=0.
\end{equation}\\

\underbar{Remark 3.} Of course, condition \eqref{eq:hyp_pour_avoir_loi_normale} implies
\begin{equation}\label{eq:condition_esperance_cond_avecF_infty_I1}
\E{f\,|\, \F_{\infty, -\infty}\vee \I_1}=0\;.
\end{equation}
In Subsection \ref{trivial-queue} we give an example where
$$
\bigcap_\ell \F_{\infty, -\ell}\vee \I_1\neq\F_{\infty, -\infty}\vee \I_1\;,
$$
showing that \eqref{eq:condition_esperance_cond_avecF_infty_I1} can be
satisfied without \eqref{eq:hyp_pour_avoir_loi_normale}. But we do
not know if \eqref{eq:condition_esperance_cond_avecF_infty_I1} is
sufficient in order to obtain the conclusion of Proposition
\ref{queue-condition}.
\begin{proof}[Proof of Proposition \ref{queue-condition}]
Define 
\begin{equation}\label{eq:def_Vmn}
 V_{m,n}=\frac 1m\sum_{i=1}^m \pr{\frac 1{\sqrt n}\sum_{j=1}^n f\circ 
T_{i,j}}^2.
\end{equation}
Following the Theorem just recalled above, it is sufficient to prove
that
$$
\lim_{n}\lim_{m} V_{m,n}=\norm{f}_2^2.
$$

By the ergodic theorem
$$
\lim_m V_{m,n}= \E{\pr{\frac 1{\sqrt n}\sum_{j=1}^n f\circ T_{0,j}}^2\mid\I_1}
$$
and the square terms give the expected limit; indeed, again by ergodic
theorem,
$$
\lim_n \E{\frac 1{n}\sum_{j=1}^n f^2\circ T_{0,j}\mid\I_1}=\E{\E{f^2\mid\I_2}\mid\I_1}
$$
and $\E{\E{f^2\mid\I_2}\mid\I_1}=\E{f^2}$ since the algebra $\I_1$ and $\I_2$ are independent.

Thus it remains to prove that
\begin{equation}\label{eq:termes-croises}
\lim_n\E{\frac1n\sum_{1\leq j< k\leq n}f\circ T_{0,j}\;f\circ T_{0,k}\mid\I_1} =0
\end{equation}
We'll need to work with order two moments of these sums, which are finite
thanks to the assumption that $f\in\mathbb{L}^4$.

Let us write, for $\ell>0$,
\begin{multline*}
  \frac1n\sum_{1\leq j< k\leq n}f\circ T_{0,j}\;f\circ T_{0,k} \\=
  \frac1n \sum_{k=\ell+1}^{n} f \circ T_{0,k}
\sum_{j=1}^{k-\ell}  f \circ T_{0,j} +
  \frac1n \sum_{k=2}^{n}  f \circ T_{0,k}
\sum_{j=(k-\ell+1)\vee 1}^{k-1} f \circ T_{0,j} =:
  I + II. 
\end{multline*}

The sequence $\left(f \circ T_{0,k} \sum_{j=(k-\ell+1)\vee 1}^{k-1} f\circ
T_{0,j}\right)_{k>1}$ is a sequence of square integrable martingale differences (remember that $f\in  \mathbb{L}^4$), adapted to the filtration $\left(\F_{0,k}\right)$.
Pythagore followed by Cauchy-Schwarz gives
$$
 \norm{II}_2^2=  \frac1{n^2} \sum_{k=2}^{n} \norm{ f \circ T_{0,k}
\sum_{j=(k-\ell+1)\vee 1}^{k-1} f \circ T_{0,j}}_2^2
  \leq \frac1{n^2} \sum_{k=2}^{n} 
\norm{f\circ T_{0,k}}_4^2\norm{\sum_{j=(k-\ell+1)\vee 1}^{k-1} f \circ T_{0,j}}_4^2
$$
Burkholder's inequality (see for example \cite{MR2472010} gives
$$
\norm{\sum_{j=(k-\ell+1)\vee 1}^{k-1} f \circ T_{0,j}}_4^2\leq3\sum_{j=(k-\ell+1)\vee 1}^{k-1}\norm{f\circ T_{0,j}}_4^2\leq 3\ell\norm{f}_4^2
$$

hence 
\begin{equation}\label{eq:controle_norme_II}
\norm{\E{II\mid\I_1}}_2\leq\norm{II}_2\leq \sqrt{3}\norm{f}_4^2\sqrt{\frac{\ell}{n}}.
\end{equation}\\

The sequences $\left(f \circ T_{0,k}
\sum_{j=1}^{k-\ell}  f \circ T_{0,j}\right)_{k>\ell}$ is a sequence of square integrable martingale differences (remember that $f\in  \mathbb{L}^4$), adapted to the filtration $\left(\F_{0,k}\right)$. Applying Proposition \ref{pro} to the $\Z$-action $(T_{k,0})$, $\I=\I_1$, $\F=\F_{\infty,k}$ and $\C$ trivial, we obtain that conditional expectations with respect to $\I_1$ and $\F_{\infty,k}$ commute. Moreover, since $\I_1\subset\F_{0,\infty}$, conditional expectations with respect to  $\I_1$ and $\F_{0,\infty}$ commute. Using the commuting property of the filtration $\left(\F_{i,j}\right)$, we affirm that conditional expectations with respect to $\I_1$ and $\F_{0,k}$ commute. As a consequence the sequence
$$
\E{\left(f \circ T_{0,k}
\sum_{j=1}^{k-\ell}  f \circ T_{0,j}\mid\I_1}\right)_{k>\ell}
$$
is a sequence of martingale differences, and in particular it is an orthogonal sequence in $ \mathbb{L}^2$. Using successively Pythagore, invariance under $T_{0,k}$, properties of the conditional expectation, Cauchy-Schwarz and Burkholder, we write

\begin{align*}
\norm{\E{I\mid\I_1}}_2^2&=  \frac1{n^2} \sum_{k=\ell+1}^{n}
  \norm{\E{f \circ T_{0,k}  \sum_{j=1}^{k-\ell}  f \circ T_{0,j}
    \mid\I_1}}_2^2\\
    &=\frac1{n^2} \sum_{k=\ell+1}^{n}
  \norm{\E{f   \sum_{j=1-k}^{-\ell}  f \circ T_{0,j}
    \mid\I_1}}_2^2\\
    &=\frac1{n^2} \sum_{k=\ell+1}^{n}
  \norm{\E{f\mid\F_{\infty,-\ell}\vee\I_1}   \sum_{j=1-k}^{-\ell}  f \circ T_{0,j}
    }_2^2\\
    &\leq\norm{\E{f\mid\F_{\infty,-\ell}\vee\I_1}}_4^2 \ \frac1n\sum_{k=\ell+1}^n\norm{\frac1{\sqrt n}\sum_{j=1-k}^{-\ell}  f \circ T_{0,j}}_4^2\\
    &\leq 3\norm{\E{f\mid\F_{\infty,-\ell}\vee\I_1}}_4^2\norm{f}_4^2.
    \end{align*}

This estimation, associated with \eqref{eq:controle_norme_II}, gives
$$
\norm{\E{\frac1n\sum_{1\leq j< k\leq n}f\circ T_{0,j}\;f\circ T_{0,k}\mid\I_1}}_2\leq\sqrt3\left(\norm{\E{f\mid\F_{\infty,-\ell}\vee\I_1}}_4\norm{f}_4+\norm{f}_4^2\sqrt{\frac{\ell}{n}}\right),
$$
hence for each $\ell$, 
\begin{equation*}
\limsup_{n\to\infty}\norm{\E{\frac1n\sum_{1\leq j< k\leq n}f\circ T_{0,j}\;f\circ T_{0,k}\mid\I_1}}_2\leq \sqrt{3}
\norm{\E{f  \,|\,\I_{1}\vee \F_{\infty, -\ell}}}_4\norm{f}_4,
\end{equation*}
and, thanks to \eqref{eq:hyp_pour_avoir_loi_normale_4}, we conclude that \eqref{eq:termes-croises} is true.
\end{proof}

\subsection{Entropy condition for convergence to a normal law}\label{entropie}

\begin{theorem}\label{pos-entropy}
There exists a martingale difference $f\in  \mathbb{L}^2$ with non normal limit in the CLT if and only if
$T_{1,0}$ is of positive entropy in $\I_2$ and $T_{0,1}$ is of positive entropy in $\I_1$.
\end{theorem}

\begin{proof}
1. Suppose that $T_{0,1}$ is of positive entropy in $\I_1$ and $T_{1,0}$ is of positive entropy in $\I_2$.
For a measure preserving and bimeasurable transformation $S$, positive
entropy implies existence of a nontrivial i.i.d.\ sequence
of the form $h\circ S^i$.  In the factor of product type given by $\I= \I_1\vee \I_2$ we thus get a non trivial field of 
Wang-Woodroofe type.

2. Suppose that the transformation $T_{0,1}$ is of zero entropy on $\I_1$. Then $\I_1$ is an invariant
sub-$\sigma$-algebra of the Pinsker sigma algebra for $T_{0,1}$ hence by Theorem 
2 in \cite{MR0928378} we have that for any integrable and 
$\F_{\infty,0}$-measurable function $g$, 
$\E{g\,|\, \F_{\infty, -\ell}\vee \I_1} = \E{g\,|\, \F_{\infty, -\ell}}$ for all $\ell >0$. Hence, if $g$ is a square integrable martingale difference, we have 
\begin{equation}\label{eq:strong_queue_condition}
\E{g\,|\, \F_{\infty, -\ell}\vee \I_1} = 0 \quad\text{for all } \ell >0. 
\end{equation}
Thus, by Proposition \ref{queue-condition}, we know that if $f$ is a martingale 
difference with finite fourth moment, then the CLT applies with a normal limit. 
For a square integrable martingale difference $f$ we will use a 
Peligrad-Voln\'y 
trick (see \cite{MR4125956}) writing $f$ as a sum of a bounded martingale 
difference and a rest small in $\mathbb L^2$. We'll conclude that the CLT 
applies to $f$ with a normal limit. Here are the details of this approximation 
argument.

Theorem 4.2 in \cite{MR0233396} tells that 
 for stochastic processes $\pr{Y_{C,n}}_{C,n\geq 1}$, 
$\pr{Y'_n}_{n\geq 1}$, $\pr{Z_C}_{C\geq 1}$, and a random variable $Z$, 
satisfying
\begin{align*}
& Y_{C,n} \longrightarrow Z_C \mbox{ in distribution, as }    n\rightarrow 
\infty, \mbox{ for each } C, \\
& Z_C\longrightarrow Z \mbox{ in distribution, as }  C\rightarrow \infty,\\
& 
\lim_{C\to\infty}\limsup_{n\to\infty}\PP\pr{\abs{Y'_n-Y_{C,n}}>\varepsilon}=0, 
\mbox{ for all } \varepsilon >0, 
\end{align*}
we may conclude that $Y'_n\to Z$ in distribution. 

Such a result extends readily to the case of a double-indexed process where the 
minimum of the indices go to infinity, by considering sequences
$\pr{m_k,n_k}$
which both go to infinity. By bounding the $\limsup$ in the last condition by a 
supremum, we thus get that for $\pr{Y_{C,m,n}}_{C,m,n\geq 1}$, 
$\pr{Y'_{m,n}}_{m,n\geq 1}$, $\pr{Z_C}_{C\geq 1}$, and a random variable $Z$, 
satisfying
\begin{align*}
& Y_{C,m,n} \longrightarrow Z_C \mbox{ in distribution, as }    
\min\ens{m,n}\rightarrow 
\infty, \mbox{ for each } C, \\
& Z_C\longrightarrow Z \mbox{ in distribution, as }  C\rightarrow \infty,\\
& 
\lim_{C\to\infty}\sup_{m,n\geq 
1}\PP\pr{\abs{Y'_{m,n}-Y_{C,m,n}}>\varepsilon}=0, 
\mbox{ for all } \varepsilon >0, 
\end{align*}
we may conclude that $Y'_{m,n}\to Z$ in distribution as $\min\ens{m,n}\to 
\infty$.

We apply this in the following setting: define 
\begin{equation*}
 f_C:=f\mathds{1}_{\abs{f}\leq C}-\E{f\mathds{1}_{\abs{f}\leq 
C}\mid\Fca_{-1,0}}-\E{f\mathds{1}_{\abs{f}\leq 
C}\mid\Fca_{0,-1}}+\E{f\mathds{1}_{\abs{f}\leq C}\mid\Fca_{-1,-1}},
\end{equation*}
\begin{equation*}
 Y'_{m,n}=\frac 1{\sqrt{mn}}\sum_{i=1}^m\sum_{j=1}^n f\circ T_{i,j}, \quad 
Y_{C,m,n}=\frac 1{\sqrt{mn}}\sum_{i=1}^m\sum_{j=1}^n f_C\circ T_{i,j}
\end{equation*}
By Proposition \ref{queue-condition}, we know that for each $C$,
$Y_{C,m,n} \longrightarrow Z_C$ in distribution as $\min\ens{m,n}\to\infty$, 
where $Z_C$ has a centered normal distribution with variance $\E{f_C^2}$. 
Moreover, $\norm{f-f_C}_2\leq 4\norm{f\mathds{1}_{\abs{f}>C}}_2$ 
hence $Z_C\longrightarrow Z$ in distribution as $C$ goes to infinity, 
where $Z$  has a centered normal distribution with variance $\E{f^2}$. 
Finally, using the fact that $f-f_C$ is a martingale difference, we get, by 
Tchebychev's inequality, 
\begin{multline*}
 \sup_{m,n\geq 
1}\PP\pr{\abs{Y'_{m,n}-Y_{C,m,n}}>\varepsilon}\\
\leq\frac 1{\varepsilon^2} \sup_{m,n\geq 
1}
\E{\pr{\frac 1{\sqrt{mn}}\sum_{i=1}^m\sum_{j=1}^n \pr{f-f_C}\circ 
T_{i,j}}^2}= \frac{1}{\eps^2}\E{\pr{f-f_C}^2},
\end{multline*}
which goes to $0$ as $C$ goes to infinity.
\end{proof}

Notice that Theorem \ref{pos-entropy} improves the result in  \cite{MR3427925}
  which tells that the limit distribution is normal as soon as one of the 
$\sigma$-algebra $\I_1$ or $\I_2$ is trivial.
However, the result in  \cite{MR3427925}
  applies to all $d>1$ while here it applies to $d=2$ only. The case of $d>2$ 
remains open.

\subsection{When the two parts are asymptotically independent}\label{convolution}

We refer here once more to Theorem 1 in \cite{MR3913270}
 which has been recalled in sub-section \ref{queue-property}.
\begin{prop} Let the characteristic functions of the limit laws for the random fields generated by $f$, $\E{f\mid \I}$ and $f-\E{f\mid \I}$ 
be respectively
$$
  \varphi_1(t) = \E {e^{-\frac12t^2\eta_1^2}}, \quad
  \varphi_2(t) = \E {e^{-\frac12t^2\eta_2^2}}, \quad
  \varphi_3(t) = \E {e^{-\frac12t^2\eta_3^2}}.
$$

If $\eta_3^2$ is a constant, in particular if \eqref{eq:hyp_pour_avoir_loi_normale} holds true, then $\varphi_1(t) = \varphi_2(t)\varphi_3(t)$, hence 
the limit law for the random field generated by $f$ is the convolution of limit laws for $\E{f\mid \I}$ and $f-\E{f\mid \I}$.
\end{prop}

\begin{proof}
Let
$$
  F_{i,v} = \frac1{\sqrt v} \sum_{j=1}^v f\circ T_{i,j}, \quad  H_{i,v} = \frac1{\sqrt v} \sum_{j=1}^v \E{f\mid\I}\circ T_{i,j}, \quad 
  \overline H_{i,v} = \frac1{\sqrt v} \sum_{j=1}^v \left(f-\E{f\mid\I}\right)\circ T_{i,j}.
$$
From  \cite{MR3913270}
  it follows that when $v\to\infty$
$$\E{F_{1,v}^2 \mid \I_1} \to \eta_1^2, \quad
\E{H_{1,v}^2 \mid \I_1} \to \eta_2^2\quad\text{and}\quad
\E{\overline H_{1,v}^2 \mid \I_1} \to \eta_3^2\;,$$
these convergences being in distribution.
By definition we have $F_{1,v} = H_{1,v} + \overline H_{1,v}$ and because $H_{1,v}$ is $\I$-measurable and $\E{\overline H_{1,v} \mid \I} = 0$, we have
$$
  \E{H_{1,v} \overline H_{1,v} \mid \I_1} = \E{ \E{H_{1,v} \overline H_{1,v} \mid \I} \mid \I_1 }= 
  \E{H_{1,v}  \E{\overline H_{1,v} \mid \I} \mid \I_1 } = 0,
$$
therefore
$$
\E{F_{1,v}^2 \mid \I_1}= \E{H_{1,v}^2 \mid \I_1}+\E{\overline H_{1,v}^2 \mid \I_1}
$$ 
and if one of these three random variables converges (in distribution) toward a constant we can affirm that
$$
  \eta_{1}^2 = \lim_{v\to\infty} \E{F_{1,v}^2 \mid \I_1} = \lim_{v\to\infty} \E{H_{1,v}^2 \mid \I_1} +  
  \lim_{v\to\infty} \E{\overline H_{1,v}^2 \mid \I_1} = \eta_{2}^2 + \eta_{3}^2.
$$
\end{proof}

\subsection{Example of a non-normal limit law, in the orthocomplement of the product factor}\label{nnormal_ex} Here we give an example of an ergodic $\Z^2$-action and a martingale difference $f$ which is orthogonal to the factor $\I$ but for which the limit distribution in the CLT is not normal.

Let us denote by $(\Omega,\B,\mu,S)$ the Bernoulli scheme
$\left(\frac12,\frac12\right)$ on the alphabet $\{-1,1\}$. This means
that $\Omega$ is the space of bilateral sequences of $-1$ or $1$,
equipped with the product $\sigma$-algebra, the probability measure
which makes the coordinate maps i.i.d.\ with law
$\left(\frac12,\frac12\right)$, and the shift $S$.

Another dynamical system is $Z=\{-1,1\}$ equipped with the permutation $U$ and the uniform probability.

On the product probability space $\Omega\times\Omega\times Z$, consider the $\Z^2$-action $T$ defined by
$$
T_{i,j}(\omega,\omega',z)=\left(S^i\omega,S^j\omega',U^{i+j}z\right).
$$
 This space is equipped with the {\it natural} filtration $\left(\F_{i,j}\right)$ :
 $$
 \F_{i,j}=\sigma\left(X_k,Y_\ell,z ; k\leq i, \ell\leq j\right)=\sigma\left(X_k ; k\leq i\right)\otimes\sigma\left(Y_\ell ;  \ell\leq j\right)\otimes\P(Z).
 $$
 where of course $X_k\left((\omega_n)_{n\in\Z}\right)=\omega_k$ et $Y_\ell\left((\omega'_m)_{m\in\Z}\right)=\omega'_\ell$ are the two independent Bernoulli processes.
 
The $\sigma$-algebra $\I_1$ of $T_{1,0}$-invariants is the $\sigma$-algebra of events depending only on the second coordinate :
$$
\I_1= (\text{trivial $\sigma$-algebra of $\Omega$})\otimes\B\otimes(\text{trivial $\sigma$-algebra of  $Z$}).
$$
And symmetrically for the $\sigma$-algebra $\I_2$ of $T_{0,1}$-invariants.
So that we have
 $$\I=\B\otimes\B\otimes(\text{trivial $\sigma$-algebra of  $Z$}).$$

We consider now the random variable $f(\omega,\omega',z)=X_0(\omega)Y_0(\omega')z$.

Since the expectation of $z$ is zero, we have $\E{f\mid\I}=0$. Associated to the $\Z^2$-action and the natural filtration, the function $f$ is a martingale difference.

But it is easy to calculate the limit distribution of $\left(\frac1{\sqrt{nm}}\sum_{i=1}^n\sum_{j=1}^m f\circ{T_{i,j}}\right)$. Indeed,
$$
\frac1{\sqrt{nm}}\sum_{i=1}^n\sum_{j=1}^m f\circ{T_{i,j}}(\omega,\omega',z)=\left(\frac1{\sqrt{n}}\sum_{i=1}^n(-1)^iX_i(\omega)\right)\left(\frac1{\sqrt{m}}\sum_{j=1}^m(-1)^jY_j(\omega')\right)z\;.
$$
The limit distribution is the distribution of the product of two independent $\mathcal N(0,1)$ random variables, which is a non-normal Bessel law.

\subsection{Remark on the asymptotic condition insuring a normal law}
\label{trivial-queue}

When looking to Proposition \ref{queue-condition} the following question appears naturally.
Is it true that
\begin{equation} \label{eq:bad_limit}
\bigcap_{\ell\geq0}\left(\F_{\infty,-\ell}\vee\I_1\right) = \F_{\infty,-\infty}\vee\I_1 \quad ?
\end{equation}

We give here a negative answer, by the construction of an example. (Note however that, in this example the transformation $T_{2,0}$ is the identity so there will not exist any non zero martingale difference.)

The space is the bidimensional torus $\T^2=\R^2/\Z^2\simeq[0,1[^2$ equipped with the Borel $\sigma$-algebra and the Lebesgue measure.
The transformation $T_{0,1}$ is the automorphism defined by the matrix
$\begin{pmatrix}
3&1\\2&1
\end{pmatrix}$,
that is $T_{0,1}(x,y)=(3x+y,2x+y)$, and the transformation $T_{1,0}$ is the translation $T_{1,0}(x,y)=\left(x+\frac12,y\right)$.

Transformations $T_{0,1}$ et $T_{1,0}$ are commuting :
$$
T_{1,0}T_{0,1}(x,y)=T_{0,1}T_{1,0}(x,y)=(3x+2y+\frac12,2x+y),
$$
so they generate a $\Z^2$-action denoted by $T$.

Let us denote by $\P$ the partition $\left\{\left[0,\frac12\right[\times[0,1[,\left[\frac12,1\right[\times[0,1[\right\}$ of the torus, and by $\left(\F_j\right)$ the filtration generated by the transformation $T_{0,1}$ and this partition :
$$
\F_j=\sigma\left(T_{0,-k}\left(\P\right),\ k\leq j\right)\;.
$$
We know that the measure preserving dynamical system $\left(\T^2,T_{0,1}\right)$ has the Kolmogorov property, thus the limit $\sigma$-algebra $\F_{-\infty}=\cap_{j\in\Z}\F_j$ is trivial (modulo the measure). 

Note also that the partition $\P$ is invariant under the transformation $T_{1,0}$.

Define, for all $i,j\in\Z$, $\F_{i,j}=\F_j$.
\\
We have $\F_{i,j}\subset\F_{i',j'}$ if $j\leq j'$.\\
We have $\F_{i,j}\cap\F_{i',j'} = \F_{(i,j)\wedge(i',j')}$ and
$
\E{\E{f\mid\F_{i,j}}\mid\F_{i',j'}} = \E{f\mid\F_{(i,j)\wedge(i',j')}}
$
for all integrable function $f$.\\
Moreover, for all $i,j\in\Z$, $\F_{i,j}=T_{-i,-j}\left(\F_{0,0}\right)$.
\\
This means that we have a completely commuting invariant filtration.

The $\sigma$-algebra $\F_{\infty,-\infty}=\F_{-\infty}$ is trivial thus $\F_{\infty,-\infty}\vee\I_1=\I_1$.

Finally the following lemma shows that property \eqref{eq:bad_limit} is not satisfied.

\begin{lem}
For all $\ell\in\Z$, the $\sigma$-algebra  $\F_{\infty,\ell}\vee\I_1$ is the whole Borel algebra.
\end{lem}
\begin{proof}
Let us show that, for any $\ell$, the $\sigma$-algebra generated by $\I_1$ and the partition $T_{0,-\ell}\left(\P\right)$ is the whole Borel algebra.

For each integer $n$, consider
$$
\begin{pmatrix}3&1\\2&1\end{pmatrix}^n=\begin{pmatrix}a_n&b_n\\c_n&d_n\end{pmatrix}
$$

A straightforward induction shows that, for all $n\in\Z$, the number $a_n$ is odd and the number $c_n$ is even.

Any measurable function of the two-dimensional variable $(2x,y)$ is $\I_1$ measurable and the map $(x,y)\mapsto\ind{\left[0,\frac12\right[}(a_\ell x+b_\ell y)$ is $T_{0,-\ell }\left(\P\right)$-measurable. By multiplication, we obtain that the map $(x,y)\mapsto\ind{\left[0,\frac12\right[}(x)$ is $T_{0,-\ell }\left(\P\right)\vee\I_1$-measurable. Using the representation $\mathbb T=[0,1[$ we write
$$
\exp(2i\pi x)=\exp\left(2i\pi\frac12(2x\text{ mod}1)\right)\ind{\left[0,\frac12\right[}(x)-\exp\left(2i\pi\frac12(2x\text{ mod}1)\right)\ind{\left[\frac12,1\right[}(x).
$$
This shows that any character of the two dimensional torus is $T_{0,-\ell }\left(\P\right)\vee\I_1$ measurable, proving that this $\sigma$-algebra is the whole Borel algebra.

\end{proof}

\underbar{Remark}. The idea behind the previous construction owes a great deal to an example attributed to Jean-Pierre Conze and communicated to us by Jean-Paul Thouvenot : if we denote by $\B(\alpha)$ the $\sigma$-algebra of Borel subsets of the one-dimensional torus invariant by the translation $x\mapsto x+\alpha$, we have that the algebra 
$
\bigcap_{n\in\N}\B(2^{-n})$ and $\bigcap_{n\in\N}\B(3^{-n})
$
are trivial (modulo the Lebesgue measure), but for each $n$ the $\sigma$-algebra
$
\B(2^{-n})\vee\B(3^{-n})
$ is the whole Borel algebra.\\
\providecommand{\bysame}{\leavevmode\hbox to3em{\hrulefill}\thinspace}
\providecommand{\MR}{\relax\ifhmode\unskip\space\fi MR }
\providecommand{\MRhref}[2]{%
  \href{http://www.ams.org/mathscinet-getitem?mr=#1}{#2}
}
\providecommand{\href}[2]{#2}

\textsc{Davide Giraudo} : Institut de Recherche Math\'ematique Avanc\'ee, UMR 7501, Universit\'e de Strasbourg and CNRS, France.

davide.giraudo@math.unistra.fr\\

\textsc{Emmanuel Lesigne} : Institut Denis Poisson, UMR 7013, Universit\'e de Tours and CNRS, France.

emmanuel.lesigne@ispoisson.fr\\

 \textsc{Dalibor Voln\'y} : Laboratoire de Math\'ematiques Rapha\"el Salem, UMR 6085, Universit\'e de Rouen Normandie and CNRS, France.

dalibor.volny@univ-rouen.fr

\end{document}